\documentclass[11pt, reqno]{amsart}
\usepackage{amsfonts}
\usepackage[T1]{fontenc}
\usepackage[dvipsnames, svgnames, table, x11names]{xcolor}
\usepackage{amsmath,amssymb,amsthm,pstricks,amscd,latexsym}
\usepackage[a4paper,centering]{geometry}
\usepackage{upgreek}
\usepackage{float}
\usepackage[pagebackref=true, colorlinks, linkcolor=myblue, anchorcolor=orange, citecolor=myred, urlcolor=mygreen, bookmarksopen, bookmarksdepth=2]{hyperref}
\usepackage[english]{babel}
\usepackage[protrusion=true, expansion=true]{microtype}
\usepackage[perpage]{footmisc}
\usepackage{stackrel}
\usepackage[pdftex]{graphicx}
\let\oldincgraphics\includegraphics
\renewcommand{\includegraphics}[1]{\oldincgraphics{img/#1}}
\usepackage{mathtools}
\usepackage{bookmark}
\usepackage{verbatim}
\usepackage{url}
\usepackage{cleveref}
\usepackage{enumitem}
\usepackage{mathabx}
\usepackage{multirow}
\usepackage{multicol}
\usepackage[all]{xy}
\usepackage[textsize=footnotesize,backgroundcolor=white]{todonotes}

\usepackage{pgfplots}
\pgfplotsset{compat=newest,compat/show suggested version=false}
\usepackage{tikz,tikz-cd,tkz-euclide}
\usetikzlibrary{fadings,patterns}
\usetikzlibrary{decorations}
\usetikzlibrary{shapes.geometric, positioning}
\usetikzlibrary{arrows, decorations.pathmorphing, decorations.pathreplacing, arrows.meta, bending, decorations.markings}
\usetikzlibrary{automata, positioning, arrows}
\usepackage{tasks}
\definecolor{emerald}{rgb}{0.31, 0.78, 0.47}
\definecolor{myblue}{rgb}{0.10, 0.10, 0.90}
\definecolor{myred}{rgb}{1.0, 0.0, 0.25}
\definecolor{mygreen}{rgb}{0.1,0.5, 0.1}

\numberwithin{equation}{section}
\newtheorem{theorem}{Theorem}[section]
\newtheorem{lemma}[theorem]{Lemma}
\newtheorem{proposition}[theorem]{Proposition}

\theoremstyle{definition}
\newtheorem{definition}[theorem]{Definition}
\newtheorem{conjecture}[theorem]{Conjecture}

\newtheorem{remark}[theorem]{Remark}

\newtheorem{condition}[theorem]{Condition}
\newcommand{\thistheoremname}{}
\newtheorem{genericthm}[theorem]{\thistheoremname}

\renewcommand{\mod}{\operatorname{mod}}
\newcommand{\proj}{\operatorname{proj}}

\newcommand{\Hom}{\operatorname{Hom}}

\newcommand{\Pot}{\operatorname{Pot}}

\renewcommand{\dim}{\operatorname{dim}}

\newcommand{\End}{\operatorname{End}}

\newcommand{\doublelrangle}[1]{\langle\hspace{-0.7mm}\langle #1\rangle\hspace{-0.7mm}\rangle}

\makeatletter
\renewcommand{\fps@figure}{htbp}
\renewcommand{\fps@table}{htbp}
\makeatother
\begin{document}
\title[Stable Brauer--Thrall II$'$ conjecture]{Stable Brauer--Thrall II$'$ conjecture for finite-dimensional Jacobian algebras}
\author{Mohamad Haerizadeh}
\address{School of Mathematics, Institute for Research in Fundamental Sciences (IPM), P.~O~.Box 19395-5746, Tehran, Iran}
\email{hyrizadeh@ipm.ir}
\author{Toshiya Yurikusa}
\address{Department of Mathematics, Osaka Metropolitan University, Osaka 558-8585, Japan}
\email{yurikusa@omu.ac.jp}
\date{\today}

\begin{abstract}
We prove that finite-dimensional Jacobian algebras associated with non-degenerate quivers with potentials satisfy the stable Brauer--Thrall II$'$ conjecture. In particular, this implies that the brick Brauer--Thrall II$'$ conjecture (also known as the $\tau$-Brauer--Thrall II$'$ conjecture) holds for finite-dimensional Jacobian algebras.
\end{abstract}

\maketitle
\tableofcontents
\section{Introduction}

Stability conditions in the sense of King \cite{Ki94} play a fundamental role in the study of finite-dimensional algebras. In this context, Pfeifer introduced in \cite{Pf25} the stable Brauer--Thrall II$'$ conjecture, which predicts that every $\tau$-tilting infinite algebra admits infinitely many isomorphism classes of $\theta$-stable modules of the same dimension vector. This conjecture strengthens an earlier brick version of Brauer--Thrall II$'$ formulated and studied in \cite[Conjecture~1.3(2)]{Mo22} and \cite[Conjecture~2]{STV21}, where $\theta$-stable modules are replaced with bricks. Despite these developments, the stable Brauer--Thrall II$'$ conjecture remains widely open in general.

In this paper, we prove the conjecture for a well-studied class of finite-dimensional algebras, namely, Jacobian algebras. They were introduced by Derksen--Weyman--Zelevinsky in \cite{DWZ08}, and form a well-behaved family of finite-dimensional algebras in the representation-theoretic approach to cluster algebras.

\begin{theorem}[{\ref{thm:conj}}]
Every finite-dimensional Jacobian algebra associated with a non-degenerate quiver with potential satisfies the stable Brauer--Thrall II$'$ conjecture. In particular, it satisfies the brick Brauer--Thrall II$'$ conjecture.
\end{theorem}

For a non-degenerate quiver with potential $(Q,W)$, every full sub-quiver $Q|_I$ with the induced potential $W|_I$ yields a natural surjective algebra homomorphism
\[
\mathcal{J}(Q,W)\twoheadrightarrow\mathcal{J}(Q|_I,W|_I).
\]
This structural property will be used in a reduction step in the proof of our main result.

A key input is a structural theorem of Seven \cite{Se07}, which asserts that any connected quiver that is not of Dynkin type contains, as a full sub-quiver, either a quiver of affine type or a generalized Kronecker quiver $K_m$ with $m\ge 3$. In particular, if $\mathcal{J}(Q,W)$ is $\tau$-tilting infinite, then one may restrict to such a full sub-quiver $Q|_I$. We show that the associated Jacobian algebra $\mathcal{J}(Q|_I,W|_I)$ already satisfies the stable Brauer--Thrall II$'$ conjecture. Since $\theta$-stability is preserved under surjective algebra homomorphisms, these $\theta$-stable modules lift to $\theta$-stable modules of $\mathcal{J}(Q,W)$. This yields the stable Brauer--Thrall II$'$ property for every finite-dimensional Jacobian algebra.

Recently, Mousavand and Paquette proved that every $E$-infinite finite-dimensional algebra admits infinitely many pairwise non-isomorphic $\theta$-stable bricks with a common dimension vector \cite[Theorem A]{MP26}. Combining their result with our characterization of $\tau$-tilting finiteness and $E$-finiteness for finite-dimensional Jacobian algebras \cite[Theorem 5.3]{HY25}, one can obtain the conclusion of our main theorem as an immediate consequence.

The paper is organized as follows. Section~\ref{sec:E} reviews $E$-invariants and the finiteness conditions arising in two-term silting theory. Section~\ref{sec:stability} recalls King's stability and introduces the stable Brauer--Thrall II$'$ condition. Section~\ref{sec:Jac} collects the background on quiver mutations, quivers with potentials, and Jacobian algebras. Finally, Section~\ref{sec:proof} establishes the stable Brauer--Thrall II$'$ conjecture for finite-dimensional Jacobian algebras.

\medskip\noindent{\bf Conventions and notations}.
Throughout this paper, $K$ denotes an algebraically closed field. We assume that all algebras are associative, unital, and basic $K$-algebras, and that all algebra homomorphisms are unital. We write $\Lambda$ for a finite-dimensional algebra. We denote by
\begin{itemize}
\item $\mod\Lambda$ the category of finite-dimensional right $\Lambda$-modules;
\item $\proj\Lambda$ the category of finite-dimensional projective right $\Lambda$-modules;
\item $K^b(\proj\Lambda)$ the homotopy category of bounded complexes in $\proj\Lambda$.
\end{itemize}
The Grothendieck group of $\mod\Lambda$ (resp.\ $K^b(\proj\Lambda)$) is denoted by $K_0(\mod\Lambda)$ (resp.\ $K_0(\proj\Lambda)$).

Given a quiver with potential $(Q,W)$, the Jacobian algebra $\mathcal{J}(Q,W)$ is defined in general as a quotient of the complete path algebra $K\doublelrangle{Q}$ by the closure of the Jacobian ideal, and may be infinite-dimensional. In this paper, we only consider quivers with potential for which $\mathcal{J}(Q,W)$ is finite-dimensional. The definition of quivers with potential and their Jacobian algebras will be recalled in Section~\ref{sec:QP}.

\section{$E$-invariants and finiteness conditions}\label{sec:E}

We recall some basic notions of two-term silting theory in the homotopy category $K^b(\proj\Lambda)$.

\begin{definition}
A complex $X\in K^b(\proj\Lambda)$ is called
\begin{itemize}
\item \emph{presilting} if $\Hom_{K^b(\proj\Lambda)}(X,X[m])=0$ for all $m>0$;
\item \emph{silting} if it is presilting and the smallest thick subcategory containing $X$ is $K^b(\proj\Lambda)$;
\item \emph{two-term} if $X^i=0$ for all $i\neq -1,0$.
\end{itemize}
\end{definition}

Note that a two-term presilting complex $X\in K^b(\proj\Lambda)$ is silting if and only if the number of non-isomorphic indecomposable direct summands of $X$ coincides with that of $\Lambda$ \cite[Proposition~2.16]{Ai13}.

\begin{definition}[{\cite[Definition 1.1]{DIJ19}}]
The algebra $\Lambda$ is called \emph{$\tau$-tilting finite} if the number of isomorphism classes of basic two-term silting complexes in $K^b(\proj\Lambda)$ is finite; otherwise it is \emph{$\tau$-tilting infinite}.
\end{definition}

We next recall the definition of the $E$-invariant, which is formulated in terms of two-term complexes in $K^b(\proj\Lambda)$. For each $g\in K_{0}(\proj\Lambda)$, there exist unique projective modules $P^{-1},P^{0}\in\proj\Lambda$ such that they do not share any non-zero direct summands and $g=[P^0]-[P^{-1}]\in K_{0}(\proj\Lambda)$. We set
\[
\Hom_{\Lambda}(g):= \Hom_{\Lambda}(P^{-1},P^{0}).
\]

Every morphism $a:P^{-1}\to P^{0}$ in $\proj\Lambda$ can be regarded as a two-term complex
\[
a=(P^{-1}\xrightarrow{a}P^{0})
\]
in $K^b(\proj\Lambda)$. For two such morphisms $a_1,a_2$, their \emph{$E$-invariant} is defined by
\[
e(a_1,a_2):=\dim_{K}\Hom_{K^b(\proj\Lambda)}(a_1,a_2[1]).
\]
For $g_1,g_2\in K_0(\proj\Lambda)$, we define the $E$-invariant of $g_1$ and $g_2$ by
\[
e(g_1,g_2):= \min\{e(a_1,a_2)\mid a_1\in\Hom_{\Lambda}(g_1),a_2\in\Hom_{\Lambda}(g_2)\}.
\footnote{Since the function $e(-,?):\Hom_{\Lambda}(g_1)\times\Hom_{\Lambda}(g_2)\rightarrow\mathbb{Z}$ is upper semicontinuous, $e(g_1,g_2)$ is equal to $e(a,b)$ for a general pair $(a_1,a_2)\in\Hom_{\Lambda}(g_1)\times\Hom_{\Lambda}(g_2)$.}
\]

\begin{definition}[{\cite[Definition~6.3]{AsIy24}}]
We say that $g \in K_0(\proj\Lambda)$ is \emph{rigid} if there exists a two-term presilting complex $X$ such that $g=[X] \in K_0(\proj\Lambda)$, and \emph{tame} if $e(g, g) = 0$. The algebra $\Lambda$ is called \emph{$E$-finite} (resp.\ \emph{$E$-tame}) if every element of $K_0(\proj\Lambda)$ is rigid (resp.\ tame).
\end{definition}

\begin{conjecture}[Demonet’s conjecture\,{\cite[Question~3.49]{De17}}]
The algebra $\Lambda$ is $\tau$-tilting finite if and only if it is $E$-finite.
\end{conjecture}

One direction of Demonet's conjecture is already known, as follows.

\begin{proposition}\label{prop:tau-E-fin}
If $\Lambda$ is $\tau$-tilting finite, then it is $E$-finite.
\end{proposition}

\begin{proof}
The assertion follows from \cite[Proposition~4.8]{As21} together with the definition of $E$-finiteness (see also \cite[Figure~1]{AsIy24}).
\end{proof}

\begin{remark}
By the above proposition, Demonet's conjecture can be stated as follows: If $\Lambda$ is $\tau$-tilting infinite, then it is not $E$-finite.
\end{remark}

\section{Stability and the stable Brauer--Thrall II$'$ conjecture}\label{sec:stability}

In this section, we recall stability conditions in the sense of King \cite{Ki94} and introduce the stable Brauer--Thrall II$'$ condition, which is the main object of study in this paper.

Using the classes of indecomposable projective and simple $\Lambda$-modules as $\mathbb{Z}$-bases, we identify $K_0(\proj\Lambda)\cong K_0(\mod\Lambda)\cong\mathbb{Z}^n$. For $M\in\mod\Lambda$, its \emph{dimension vector} is the class $[M]\in K_0(\mod\Lambda)\cong\mathbb{Z}^n$. We denote by $\langle-,-\rangle$ the canonical pairing between $K_0(\proj\Lambda)$ and $K_0(\mod\Lambda)$.

\begin{definition}[{\cite{Ki94}}]
Let $\theta\in K_0(\proj\Lambda)$. A module $M\in\mod\Lambda$ is called
\begin{itemize}
\item \emph{$\theta$-semistable} if $\langle\theta,[M]\rangle = 0$ and $\langle\theta,[M']\rangle\le 0$ for all submodules $M'\subseteq M$;
\item \emph{$\theta$-stable} if $\langle\theta,[M]\rangle = 0$ and $\langle\theta,[M']\rangle < 0$ for all proper submodules $M'\subset M$.
\end{itemize}
\end{definition}

The full subcategory of $\mod\Lambda$ consisting of $\theta$-semistable modules is defined by
\[
\mathcal{W}_{\theta}^{\Lambda}:=\{M\in\mod\Lambda\mid \text{$M$ is $\theta$-semistable}\}.
\]

Stability conditions allow us to formulate the following condition for $\Lambda$.

\begin{condition}\label{cond:sBT}
There exist a dimension vector $\mathbf{d}$ and $\theta\in K_0(\proj\Lambda)$ such that there are infinitely many isomorphism classes of $\theta$-stable $\Lambda$-modules of dimension vector $\mathbf{d}$.
\end{condition}

The following Brauer--Thrall-type conjecture, posed by Pfeifer in \cite[Conjecture~5.2]{Pf25}, strengthens the brick Brauer--Thrall II$'$ conjecture (also known as the $\tau$-Brauer--Thrall II$'$ conjecture) formulated and studied in \cite[Conjecture~1.3(2)]{Mo22} and \cite[Conjecture~2]{STV21}.

\begin{conjecture}[Stable Brauer--Thrall II$'$ conjecture]
If $\Lambda$ is $\tau$-tilting infinite, then it satisfies Condition~\ref{cond:sBT}.
\end{conjecture}

Recall that a module $M\in\mod\Lambda$ is called a \emph{brick} if $\End_{\Lambda}(M)$ is a division algebra. Over an algebraically closed field $K$, this is equivalent to $\End_{\Lambda}(M)\cong K$.

\begin{conjecture}[{Brick Brauer--Thrall II$'$ conjecture}]
If $\Lambda$ is $\tau$-tilting infinite, then there exists a dimension vector $\mathbf{d}$ such that there are infinitely many isomorphism classes of bricks $M\in\mod\Lambda$ of dimension vector $\mathbf{d}$.
\end{conjecture}

\begin{remark}\label{rem: st imp brick conj}
Since $\theta$-stable modules are bricks, the stable Brauer--Thrall II$'$ conjecture implies the brick Brauer--Thrall II$'$ conjecture.
\end{remark}

Pfeifer proved the following fundamental relation among the conjectures.

\begin{theorem}[{\cite[Main Theorem~1]{Pf25}}]\label{thm:Pf}
If $\Lambda$ satisfies the stable Brauer--Thrall II$'$ conjecture, then it satisfies Demonet’s conjecture. Moreover, if $\Lambda$ is $E$-tame, then the converse holds.
\end{theorem}

\begin{lemma}\label{lem:sBT E-tame}
If $\Lambda$ is $E$-tame but not $E$-finite, then it satisfies Condition~\ref{cond:sBT}.
\end{lemma}

\begin{proof}
Since $\Lambda$ is not $E$-finite, it is not $\tau$-tilting finite by Proposition~\ref{prop:tau-E-fin}. Hence $\Lambda$ satisfies Demonet's conjecture. Since it is $E$-tame, it also satisfies the stable Brauer--Thrall II$'$ conjecture by Theorem~\ref{thm:Pf}. Therefore, it satisfies Condition~\ref{cond:sBT}.
\end{proof}

In the rest of this section, we prove the following:

\begin{proposition}\label{prop:sBT pres}
Let $\Lambda\twoheadrightarrow\Lambda'$ be a surjective algebra homomorphism. If $\Lambda'$ satisfies Condition~\ref{cond:sBT}, then so does $\Lambda$.
\end{proposition}

First, the following lemma is standard (see \cite[Proof of Proposition~3.1]{Ki94}).

\begin{lemma}[{\cite{Ki94}}]\label{lem:stable=simple}
Let $\theta\in K_0(\proj\Lambda)$. Then a module $M\in\mod\Lambda$ is $\theta$-stable if and only if it is a simple object of $\mathcal{W}_{\theta}^{\Lambda}$.
\end{lemma}

A surjective algebra homomorphism $\varphi:\Lambda\twoheadrightarrow\Lambda'$ induces a fully faithful functor $\mod\Lambda'\rightarrow\mod\Lambda$, which can be viewed as an inclusion functor: every $\Lambda'$-module is regarded as a $\Lambda$-module via $\varphi$.

\begin{lemma}[{\cite[Example 3.23]{AsIy24}}]
For a surjective algebra homomorphism $\Lambda\twoheadrightarrow\Lambda'$, the induced group homomorphism
\begin{equation}\label{eq:quotient Grothen}
-\otimes_{\Lambda}\Lambda':K_0(\proj\Lambda)\longrightarrow K_0(\proj\Lambda')
\end{equation}
is surjective. Moreover, for each $\theta\in K_{0}(\proj\Lambda)$, we have
\begin{equation}\label{eq:quotient ss}
\mathcal{W}^{\Lambda'}_{\theta\otimes_{\Lambda}\Lambda'}=\mathcal{W}^{\Lambda}_{\theta}\cap\mod\Lambda'.
\end{equation}
\end{lemma}

\begin{proposition}\label{prop:quotient stable}
Let $\Lambda\twoheadrightarrow\Lambda'$ be a surjective algebra homomorphism, and $\theta\in K_{0}(\proj\Lambda)$. Then a $\theta\otimes_{\Lambda}\Lambda'$-stable $\Lambda'$-module is also a $\theta$-stable $\Lambda$-module.
\end{proposition}

\begin{proof}
Let $M$ be a $\theta\otimes_{\Lambda}\Lambda'$-stable $\Lambda'$-module. Then $M\in\mathcal{W}_{\theta}^{\Lambda}$ by \eqref{eq:quotient ss}. Consider a $\Lambda$-submodule $M'$ of $M$ that is a simple object of $\mathcal{W}_{\theta}^{\Lambda}$. Then by Lemma~\ref{lem:stable=simple}, $M'$ is a $\theta$-stable $\Lambda$-module. Since $M$ is a $\Lambda'$-module, any $\Lambda$-submodule of $M$, and in particular $M'$, is also annihilated by $\ker(\Lambda\twoheadrightarrow\Lambda')$. Thus $M'\in\mod\Lambda'$, and hence $M'\in\mathcal{W}^{\Lambda'}_{\theta\otimes_{\Lambda}\Lambda'}$ by \eqref{eq:quotient ss}. However, by Lemma~\ref{lem:stable=simple}, $M$ is a simple object of $\mathcal{W}^{\Lambda'}_{\theta\otimes_{\Lambda}\Lambda'}$, and this implies that $M=M'$. Therefore, $M$ is a $\theta$-stable $\Lambda$-module.
\end{proof}

\begin{proof}[Proof of Proposition~\ref{prop:sBT pres}]
Since the group homomorphism \eqref{eq:quotient Grothen} is surjective, there exist a dimension vector $\mathbf{d}'\in K_0(\mod\Lambda')$ and $\theta\in K_0(\proj\Lambda)$ such that there are infinitely many isomorphism classes of $\theta\otimes_{\Lambda}\Lambda'$-stable $\Lambda'$-modules of dimension vector $\mathbf{d}'$. 

By Proposition~\ref{prop:quotient stable}, all of them are $\theta$-stable $\Lambda$-modules with the same dimension. Since the set of possible dimension vectors is finite, we conclude that there exists a dimension vector $\mathbf{d}$ for which there are infinitely many isomorphism classes of $\theta$-stable $\Lambda$-modules of dimension vector $\mathbf{d}$. Therefore, $\Lambda$ satisfies Condition~\ref{cond:sBT}.
\end{proof}

\section{Jacobian algebras}\label{sec:Jac}

\subsection{Quiver mutations}\label{sec:mutation}

Let $Q$ be a quiver without loops or $2$-cycles (i.e., oriented cycles of length two). The \emph{mutation} of $Q$ at vertex $k$ is the quiver $\mu_k(Q)$ obtained from $Q$ by the following steps:
\begin{itemize}
\item[$(1)$] For each path $i\leftarrow k\leftarrow j$, add an arrow $i\leftarrow j$.
\item[$(2)$] Reverse all arrows incident to $k$.
\item[$(3)$] Remove a maximal set of disjoint $2$-cycles.
\end{itemize}
Note that $\mu_k$ is an involution, that is, $\mu_k(\mu_k(Q))=Q$.

\begin{definition}
The quiver $Q$ is called
\begin{itemize}
\item \emph{mutation equivalent} to a quiver $Q'$ if it can be obtained from $Q'$ by a finite sequence of mutations;
\item \emph{of Dynkin} (resp.\ \emph{affine}) \emph{type} if it is mutation equivalent to a Dynkin (resp.\ affine\footnote{The $2$-Kronecker quiver is regarded as of type $A^{(1)}_1$.}) quiver (see Figures~\ref{fig:Dynkin} and~\ref{fig:affine});
\item \emph{mutation-acyclic} if it is mutation equivalent to an acyclic quiver.
\end{itemize}
\end{definition}

\begin{figure}[htp]
\centering
\begin{tikzcd}[cramped,sep=tiny]
{A_{n\ge 1}:} & \bullet & \bullet & \bullet & \cdots & \bullet & \bullet \\
& \bullet \\
{D_{n\ge 4}:} && \bullet & \bullet & \cdots & \bullet & \bullet \\
& \bullet \\
{E_{n=6,7,8}:} & \bullet & \bullet & \bullet & \bullet & \cdots & \bullet \\
&&& \bullet
\arrow[no head, from=1-2, to=1-3]
\arrow[no head, from=1-3, to=1-4]
\arrow[no head, from=1-4, to=1-5]
\arrow[no head, from=1-5, to=1-6]
\arrow[no head, from=1-6, to=1-7]
\arrow[no head, from=2-2, to=3-3]
\arrow[no head, from=3-3, to=3-4]
\arrow[no head, from=3-4, to=3-5]
\arrow[no head, from=3-5, to=3-6]
\arrow[no head, from=3-6, to=3-7]
\arrow[no head, from=4-2, to=3-3]
\arrow[no head, from=5-2, to=5-3]
\arrow[no head, from=5-3, to=5-4]
\arrow[no head, from=5-4, to=5-5]
\arrow[no head, from=5-4, to=6-4]
\arrow[no head, from=5-5, to=5-6]
\arrow[no head, from=5-6, to=5-7]
\end{tikzcd}
\caption{Underlying diagrams of Dynkin quivers with $n$ vertices}
\label{fig:Dynkin}
\end{figure}
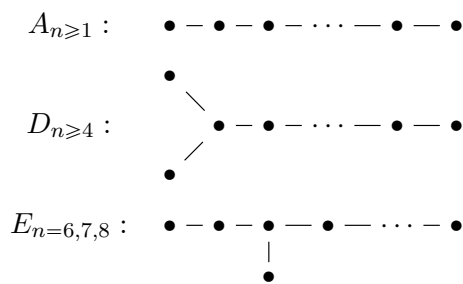

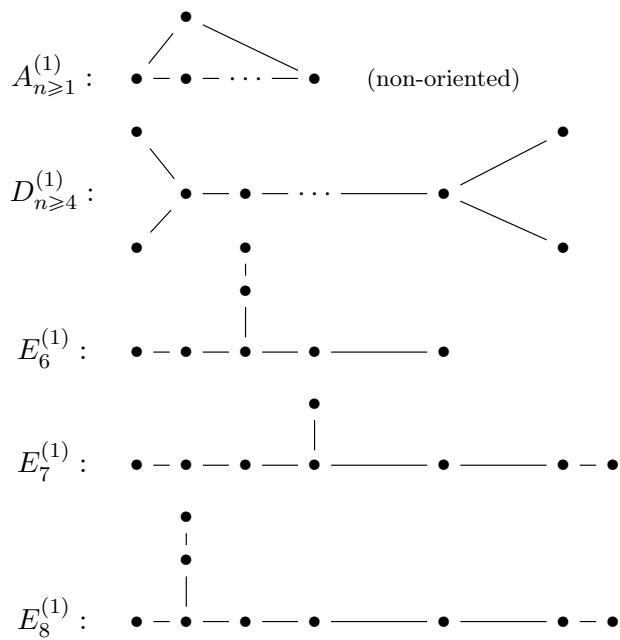
\begin{figure}[htp]
\centering
\begin{tikzcd}[cramped,sep=tiny]
&& \bullet \\
{A^{(1)}_{n\ge 1}:} & \bullet & \bullet & \cdots & \bullet & {\text{{\footnotesize (non-oriented)}}} \\
& \bullet &&&&& \bullet \\
{D^{(1)}_{n\ge 4}:} && \bullet & \bullet & \cdots & \bullet \\
& \bullet && \bullet &&& \bullet \\
&&& \bullet \\
{E^{(1)}_6:} & \bullet & \bullet & \bullet & \bullet & \bullet \\
&&&& \bullet \\
{E_7^{(1)}:} & \bullet & \bullet & \bullet & \bullet & \bullet & \bullet & \bullet \\
&& \bullet \\
&& \bullet \\
{E^{(1)}_8:} & \bullet & \bullet & \bullet & \bullet & \bullet & \bullet & \bullet
\arrow[no head, from=1-3, to=2-2]
\arrow[no head, from=1-3, to=2-5]
\arrow[no head, from=2-2, to=2-3]
\arrow[no head, from=2-3, to=2-4]
\arrow[no head, from=2-4, to=2-5]
\arrow[no head, from=3-2, to=4-3]
\arrow[no head, from=4-3, to=4-4]
\arrow[no head, from=4-4, to=4-5]
\arrow[no head, from=4-5, to=4-6]
\arrow[no head, from=4-6, to=3-7]
\arrow[no head, from=4-6, to=5-7]
\arrow[no head, from=5-2, to=4-3]
\arrow[no head, from=5-4, to=6-4]
\arrow[no head, from=6-4, to=7-4]
\arrow[no head, from=7-3, to=7-2]
\arrow[no head, from=7-4, to=7-3]
\arrow[no head, from=7-4, to=7-5]
\arrow[no head, from=7-5, to=7-6]
\arrow[no head, from=8-5, to=9-5]
\arrow[no head, from=9-2, to=9-3]
\arrow[no head, from=9-3, to=9-4]
\arrow[no head, from=9-4, to=9-5]
\arrow[no head, from=9-5, to=9-6]
\arrow[no head, from=9-6, to=9-7]
\arrow[no head, from=9-7, to=9-8]
\arrow[no head, from=10-3, to=11-3]
\arrow[no head, from=11-3, to=12-3]
\arrow[no head, from=12-2, to=12-3]
\arrow[no head, from=12-3, to=12-4]
\arrow[no head, from=12-4, to=12-5]
\arrow[no head, from=12-5, to=12-6]
\arrow[no head, from=12-6, to=12-7]
\arrow[no head, from=12-7, to=12-8]
\end{tikzcd}
\caption{Underlying diagrams of affine quivers with $n+1$ vertices}
\label{fig:affine}
\end{figure}

\begin{lemma}\label{lem:nDynaffsub}
If a connected quiver without loops or $2$-cycles is not of Dynkin type, then it contains a quiver of affine type or $K_m$ with $m\ge 3$ as a full sub-quiver.
\end{lemma}

\begin{proof}
The assertion follows from \cite[Theorem~3.2 and (2.2)]{Se07}.
\end{proof}

\subsection{Quivers with potentials and Jacobian algebras}\label{sec:QP}

Recall the notions of quivers with potentials and Jacobian algebras \cite{DWZ08}. Let $Q$ be a quiver. For each $m\in\mathbb{Z}_{\ge 0}$, we denote by $Q_{m}$ the set of all paths of length $m$. Consider the $K$-vector space $KQ_m$ freely generated by $Q_m$. Then the \emph{complete path algebra} associated with $Q$ is defined by
\[
K\doublelrangle{Q}:=\prod_{m=0}^{\infty}KQ_m,
\]
where multiplication is induced by the concatenation of paths. Let $\mathfrak{m}$ be the ideal of $K\doublelrangle{Q}$ generated by all arrows. For a subset $U\subset K\doublelrangle{Q}$, the \emph{$\mathfrak{m}$-adic closure} of $U$ is denoted by
\[
\overline{U}:=\bigcap_{l=0}^{\infty}U+\mathfrak{m}^{l}.
\]

Assume that $Q$ has no loops. A \emph{potential} $W$ on $Q$ is a (possibly infinite) linear combination of oriented cycles in $Q$ modulo the closure of the subspace $U$ generated by the commutators in $[K\doublelrangle{Q},K\doublelrangle{Q}]$ and trivial paths $e_i$, for any $i\in Q_0$. Indeed, $W$ is an element of 
\[
\Pot(Q):=K\doublelrangle{Q}/\overline{U}.
\]
The pair $(Q, W)$ is called a \emph{quiver with potential} or just a \emph{QP}, for short. For an arrow $\alpha$ of $Q$, we define a \emph{cyclic derivative} $\partial_{\alpha}:\Pot(Q)\rightarrow  K\doublelrangle{Q}$ as the unique continuous linear map that sends any cycle $c$ to $\sum\limits_{c=u\alpha v}vu$.

\begin{definition}
The \emph{Jacobian algebra} associated with the QP $(Q, W)$ is defined by
\[
\mathcal{J}(Q,W):=K\doublelrangle{Q}/\overline{J(W)},
\]
where $J(W):=\langle \partial_{\alpha}W\mid\alpha\in Q_1\rangle$ is an ideal of $K\doublelrangle{Q}$.
\end{definition}

Let $(Q, W)$ and $(Q', W')$ be two QPs with $Q_0=Q'_0$. Their \emph{direct sum} is a QP $(Q'',W'')$ defined by $Q''_0=Q_0=Q'_0$, $Q''_1=Q_1\sqcup Q'_1$, and $W''=W+W'$. Moreover, we say that they are \emph{right equivalent} if there is an algebra isomorphism $\varphi:K\doublelrangle{Q}\rightarrow K\doublelrangle{Q'}$ whose restriction on vertices is the identity map and $\varphi(W)=W'$.

We recall two special classes of QPs: A QP $(Q,W)$ is called \emph{trivial} if $W$ is a linear combination of $2$-cycles and $\mathcal{J}(Q,W)$ is isomorphic to $KQ_{0}$. It is called \emph{reduced} if $W$ has no terms that are cycles of length at most $2$.

\begin{theorem}[{\cite[Theorem~4.6]{DWZ08}}]
Let $(Q, W)$ be a QP. Up to right equivalence, uniquely, we can write it as a direct sum of a trivial QP $(Q_{\text{tri}}, W_{\text{tri}})$ and a reduced QP $(Q_{\text{red}}, W_{\text{red}})$, called the \emph{reduced part} of $(Q, W)$. Moreover, $\mathcal{J}(Q,W)\cong\mathcal{J}(Q_{\text{red}},W_{\text{red}})$.
\end{theorem}

Let $(Q, W)$ be a QP. Assume that $k\in Q_0$ lies on no $2$-cycles. Moreover, we may assume that no oriented cycle appearing in the expansion of $W$ starts or ends at $k$. The \emph{QP-premutation} of $(Q,W)$ at $k$ is the QP $\tilde{\mu}_k(Q,W):=(\tilde{Q},\tilde{W})$ obtained from $(Q,W)$ by the following steps (cf.\ Section~\ref{sec:mutation}):
\begin{itemize}
\item[$(1)$] For each
$(b,a)\in Q_{2,k}:=
\{(b,a)\in Q_1\times Q_1\mid\begin{tikzcd}[cramped]
i & k & j
\arrow["b", from=1-2, to=1-1,swap]
\arrow["a", from=1-3, to=1-2,swap]
\end{tikzcd}\}$,
add an arrow 
$\begin{tikzcd}[cramped]
i & j
\arrow["{[ba]}", from=1-2, to=1-1,swap]
\end{tikzcd}$.
\item[$(2)$] Replace each arrow
$\begin{tikzcd}[cramped]
k & i
\arrow["{a}", from=1-2, to=1-1,swap]
\end{tikzcd}$
with
$\begin{tikzcd}[cramped]
k & i
\arrow["{a^{\ast}}", from=1-1, to=1-2]
\end{tikzcd}$, and each arrow
$\begin{tikzcd}[cramped]
i & k
\arrow["{a}", from=1-2, to=1-1,swap]
\end{tikzcd}$
with
$\begin{tikzcd}[cramped]
i & k
\arrow["{a^{\ast}}", from=1-1, to=1-2]
\end{tikzcd}$.
\item[$(3)$] Define $\tilde{W}$ by \[\tilde{W}:=[W]+\sum_{(b,a)\in Q_{2,k}} [ba]a^{\ast}b^{\ast}.\] Here, $[W]$ is obtained from $W$ by replacing each occurrence of $ba$, where $(b,a)\in Q_{2,k}$, with the arrow $[ba]$.
\end{itemize}
Up to right equivalence, the reduced part of $\tilde{\mu}_k(Q, W)$ is called the \emph{QP-mutation} of $(Q, W)$ at $k$, which is denoted by $\mu_k(Q, W)$.

\begin{definition}
We say that $W$ or $(Q, W)$ is
\begin{itemize}
\item \emph{non-degenerate} if every quiver obtained from $(Q, W)$ by a finite sequence of mutations has no $2$-cycles;
\item \emph{Jacobi-finite} if the associated Jacobian algebra $\mathcal{J}(Q,W)$ is finite-dimensional.
\end{itemize}
\end{definition}

Note that for a non-degenerate QP $(Q,W)$, the quiver of $\mu_k(Q,W)$ is equal to $\mu_k(Q)$ for $k\in Q_0$. If $Q$ is a quiver without loops or $2$-cycles and $K$ is uncountable, then $Q$ admits a non-degenerate potential \cite[Corollary~7.4]{DWZ08}.

We now recall a fundamental finiteness criterion for finite-dimensional Jacobian algebras associated with non-degenerate QPs, which will be used in the next section.

\begin{theorem}[{\cite[Theorem~5.3]{HY25}}]\label{thm:fin TFAE}
For a connected non-degenerate Jacobi-finite QP $(Q,W)$, the Jacobian algebra $\mathcal{J}(Q,W)$ is $\tau$-tilting finite if and only if it is $E$-finite, and this holds precisely when $Q$ is of Dynkin type. In particular, $\mathcal{J}(Q,W)$ satisfies Demonet's conjecture.
\end{theorem}

Finally, we recall the notion of restriction of a QP, which plays an important role in the proof of our main theorem. Let $(Q,W)$ be a QP, and $I$ be a subset of $Q_{0}$. Then the \emph{restriction} of $(Q,W)$ to $I$ is denoted by $(Q|_I,W|_I)$ and defined as follows:
\begin{itemize}
\item $Q|_I$ is the full sub-quiver of $Q$ with $(Q|_I)_0=I$.
\item $W|_I$ is obtained from $W$ by deleting summands of $W$ that are not cycles in $Q|_I$.
\end{itemize}

\begin{proposition}\label{prop:restrict}
Let $(Q,W)$ be a QP and $I\subset Q_0$.
\begin{itemize}
\item[$(1)$]\cite[Proposition~8.9]{DWZ08} The restriction of $(Q,W)$ to $I$ induces a surjective algebra homomorphism $\mathcal{J}(Q,W)\twoheadrightarrow\mathcal{J}(Q|_I,W|_I)$.
\item[$(2)$]\cite[Corollary~22]{LF09} If $(Q,W)$ is non-degenerate, then so is $(Q|_I,W|_I)$.
\end{itemize}
\end{proposition}

\section{The stable Brauer--Thrall II$'$ conjecture for Jacobian algebras}\label{sec:proof}

In this section, we prove the following theorem, which is the main result of the paper.

\begin{theorem}\label{thm:conj}
Let $(Q,W)$ be a non-degenerate quiver with potential. If $\mathcal{J}(Q,W)$ is finite-dimensional, then it satisfies the stable Brauer--Thrall II$'$ conjecture. In particular, it satisfies the brick Brauer--Thrall II$'$ conjecture.
\end{theorem}

To prove Theorem~\ref{thm:conj}, we establish two lemmas.

\begin{lemma}\label{lem:affine E}
Let $(Q, W)$ be a non-degenerate QP. If $Q$ is mutation-acyclic, then $\mathcal{J}(Q,W)$ is finite-dimensional. In particular, if $Q$ is of affine type, then $\mathcal{J}(Q,W)$ is $E$-tame but not $E$-finite.
\end{lemma}

\begin{proof}
If $Q'$ is an acyclic quiver, then any non-degenerate potential on $Q'$ is zero and $\mathcal{J}(Q',0)$ is just a finite-dimensional path algebra $KQ'$. Since the finite-dimensionality of Jacobian algebras is preserved under QP-mutation \cite[Corollary~6.6]{DWZ08}, $\mathcal{J}(Q,W)$ is also finite-dimensional. The remaining assertions follow immediately from Theorem~\ref{thm:fin TFAE} and \cite[Theorem~5.5]{HY25}.
\end{proof}

\begin{lemma}\label{lem:sBT}
Let $(Q, W)$ be a non-degenerate Jacobi-finite QP. If $Q$ is not of Dynkin type, then $\mathcal{J}(Q,W)$ satisfies Condition~\ref{cond:sBT}.
\end{lemma}

\begin{proof}
By Lemma~\ref{lem:nDynaffsub}, there exists a subset $I\subset Q_0$ such that $Q|_I$ is either of affine type or $K_m$ with $m\ge3$. Moreover, Proposition~\ref{prop:restrict} and Lemma~\ref{lem:affine E} imply that
\begin{itemize}
\item there exists a surjective algebra homomorphism $\mathcal{J}(Q,W)\twoheadrightarrow\mathcal{J}(Q|_I,W|_I)$;
\item $(Q|_I,W|_I)$ is non-degenerate;
\item $\mathcal{J}(Q|_I,W|_I)$ is finite-dimensional.
\end{itemize}

If $Q|_I$ is of affine type, then $\mathcal{J}(Q|_I,W|_I)$ is $E$-tame but not $E$-finite by Lemma~\ref{lem:affine E}. Hence, by Lemma~\ref{lem:sBT E-tame}, it satisfies Condition~\ref{cond:sBT}.

For the case where $Q|_I$ is $K_m$ with $m\ge3$, note that $\mathcal{J}(K_m,0)$ admits a surjective homomorphism onto $\mathcal{J}(A_1^{(1)},0)$ (obtained by killing all but two arrows). Since $\mathcal{J}(A_1^{(1)},0)$ satisfies Condition~\ref{cond:sBT} by the previous case, Proposition~\ref{prop:sBT pres} yields that $\mathcal{J}(K_m,0)$ also satisfies Condition~\ref{cond:sBT}.

Finally, applying Proposition~\ref{prop:sBT pres} to the surjection $\mathcal{J}(Q,W)\twoheadrightarrow\mathcal{J}(Q|_I,W|_I)$ shows that $\mathcal{J}(Q,W)$ satisfies Condition~\ref{cond:sBT}.
\end{proof}

Now we are ready to prove Theorem~\ref{thm:conj}.

\begin{proof}[Proof of Theorem~\ref{thm:conj}]
The first assertion follows from Theorem~\ref{thm:fin TFAE} and Lemma~\ref{lem:sBT}. The second one follows from the first part together with Remark~\ref{rem: st imp brick conj}.
\end{proof}

\medskip\noindent{\bf Acknowledgements}.
The authors would like to thank Kaveh Mousavand, Charles Paquette, Calvin Pfeifer, and Pavel Tumarkin for several helpful conversations. The research of the first author was in part supported by a grant from IPM. The second author was supported by JSPS KAKENHI Grant Number JP21K13761.

\bibliographystyle{alpha}
\bibliography{all}

@Article{Ki94,
  author    = {Alastair D. King},
  journal   = {The Quarterly Journal of Mathematics},
  title     = {Moduli of representations of finite dimensional algebras},
  year      = {1994},
  issn      = {1464-3847},
  number    = {4},
  pages     = {515--530},
  volume    = {45},
  doi       = {10.1093/qmath/45.4.515},
  issue     = {4},
  publisher = {Oxford University Press (OUP)},
  url       = {https://academic.oup.com/qjmath/article-lookup/doi/10.1093/qmath/45.4.515},
}

@Article{DWZ08,
  author    = {Harm Derksen and Jerzy Weyman and Andrei Zelevinsky},
  journal   = {Selecta Mathematica},
  title     = {Quivers with potentials and their representations {{I}}: Mutations},
  year      = {2008},
  issn      = {1420-9020},
  month     = jun,
  number    = {1},
  pages     = {59--119},
  volume    = {14},
  doi       = {10.1007/s00029-008-0057-9},
  issue     = {1},
  publisher = {Springer Science and Business Media LLC},
  url       = {http://link.springer.com/10.1007/s00029-008-0057-9},
}

@Article{Ai13,
  author    = {Aihara, Takuma},
  journal   = {Algebras and Representation Theory},
  title     = {Tilting-Connected Symmetric Algebras},
  year      = {2013},
  issn      = {1572-9079},
  month     = may,
  number    = {3},
  pages     = {873--894},
  volume    = {16},
  doi       = {10.1007/s10468-012-9337-3},
  issue     = {3},
  publisher = {Springer Science and Business Media LLC},
  url       = {http://link.springer.com/10.1007/s10468-012-9337-3},
}

@Article{DIJ19,
  author    = {Demonet, Laurent and Iyama, Osamu and Jasso, Gustavo},
  journal   = {International Mathematics Research Notices},
  title     = {$\tau$-Tilting Finite Algebras, Bricks, and g-Vectors},
  year      = {2019},
  issn      = {1687-0247},
  month     = jul,
  number    = {3},
  pages     = {852--892},
  volume    = {2019},
  doi       = {10.1093/imrn/rnx135},
  issue     = {3},
  publisher = {Oxford University Press (OUP)},
  url       = {http://arxiv.org/abs/1503.00285 http://dx.doi.org/10.1093/imrn/rnx135},
}

@Article{As21,
  author    = {Asai, Sota},
  journal   = {Advances in Mathematics},
  title     = {The wall-chamber structures of the real {{Grothendieck}} groups},
  year      = {2021},
  issn      = {0001-8708},
  month     = apr,
  pages     = {107615},
  volume    = {381},
  doi       = {10.1016/j.aim.2021.107615},
  publisher = {Elsevier BV},
  url       = {https://linkinghub.elsevier.com/retrieve/pii/S0001870821000530},
}

@Article{AsIy24,
  author    = {Asai, Sota and Iyama, Osamu},
  journal   = {Proceedings of the London Mathematical Society},
  title     = {Semistable torsion classes and canonical decompositions in {{Grothendieck}} groups},
  year      = {2024},
  issn      = {1460-244X},
  month     = oct,
  number    = {5},
  volume    = {129},
  copyright = {arXiv.org perpetual, non-exclusive license},
  doi       = {10.1112/plms.12639},
  publisher = {Wiley},
}

@Article{Pf25,
  author    = {Pfeifer, Calvin},
  journal   = {Bulletin of the London Mathematical Society},
  title     = {Remarks on $\tau$-tilted versions of the second {{Brauer–Thrall}} conjecture},
  year      = {2025},
  number    = {5},
  pages     = {1568--1583},
  volume    = {57},
  publisher = {Wiley Online Library},
}

@Article{LF09,
  author    = {Labardini-Fragoso, Daniel},
  journal   = {Proceedings of the London Mathematical Society},
  title     = {Quivers with potentials associated to triangulated surfaces},
  year      = {2008},
  issn      = {0024-6115},
  month     = nov,
  number    = {3},
  pages     = {797--839},
  volume    = {98},
  doi       = {10.1112/plms/pdn051},
  publisher = {Wiley},
}

@Article{HY25,
  author        = {Mohamad Haerizadeh and Toshiya Yurikusa},
  journal       = {arXiv:2507.04570},
  title         = {Finite-dimensional Jacobian algebras: Finiteness and tameness},
  year          = {2025},
  archiveprefix = {arXiv},
  eprint        = {2507.04570},
  primaryclass  = {math.RT},
  url           = {https://arxiv.org/abs/2507.04570},
}

@Article{De17,
  author    = {Demonet, Laurent},
  journal   = {l'Habilitation \`{a} Diriger des Recherches, Laboratoire de Math\'{e}atiques Nicolas Oresme, Universit\'{e} de Caen},
  title     = {Combinatorics of mutations in representation theory},
  year      = {2017},
  date      = {2017},
  publisher = {Habilitation},
  school    = {Laboratoire de Math\'{e}atiques Nicolas Oresme, Universit\'{e} de Caen},
  url       = {http://www.math.nagoya-u.ac.jp/~demonet/recherche/habilitation.pdf},
}

@Article{Se07,
  author   = {Seven, Ahmet},
  journal  = {The Electronic Journal of Combinatorics [electronic only]},
  title    = {Recognizing cluster algebras of finite type.},
  year     = {2007},
  volume   = {14},
  language = {en},
  url      = {http://dml.mathdoc.fr/item/05132837},
  zbl      = {1114.05103},
}

@article{Mo22,
  title={$\tau$-tilting finiteness of non-distributive algebras and their module varieties},
  author={Mousavand, Kaveh},
  journal={Journal of Algebra},
  volume={608},
  pages={673--690},
  year={2022},
  publisher={Elsevier}
}

@article{STV21,
  title={On band modules and $\tau$-tilting finiteness},
  author={Schroll, Sibylle and Treffinger, Hipolito and Valdivieso, Yadira},
  journal={Mathematische Zeitschrift},
  volume={299},
  number={3},
  pages={2405--2417},
  year={2021},
  publisher={Springer}
}

@misc{MP26,
	author        = {Mousavand, Kaveh and Paquette, Charles},
	title         = {A Proof of the Second brick-Brauer-Thrall
	Conjecture for {$E$}-infinite algebras},
	year          = {2026},
	eprint        = {2609.31417},
	archivePrefix = {arXiv},
	primaryClass  = {math.RT},
	url           = {https://arxiv.org/abs/2609.31417}
}
\end{document}